\documentclass[12pt,twoside,reqno]{amsart}
\usepackage{pifont}
\usepackage{amsmath}
\usepackage{amsthm}
\usepackage{amsfonts}
\usepackage{amssymb}
\usepackage{latexsym}
\usepackage{amstext}
\usepackage{array}
\usepackage{graphicx}
\usepackage{amsthm}

\usepackage{amsthm}

\date{}
\allowdisplaybreaks[4] \footskip=15pt
\renewcommand{\uppercasenonmath}[1]{}

\newtheorem{thm}[subsection]{Theorem}
\newtheorem{cor}[subsection]{Corollary }
\newtheorem{Def}[subsection]{Definition}
\newtheorem{lem}[subsection]{Lemma}
\newtheorem{remark}[subsection]{Remark}

\newtheorem{prop}[subsection]{Proposition}
\newtheorem{exm}[subsection]{Example}
  
\newcommand{\bthm}{\begin{thm} }
\newcommand{\ethm}{\end{thm} }
\newcommand{\bpro}{\begin{prop}}
\newcommand{\epro}{\end{prop}}
\newcommand{\bdf}{\begin{Def}}
\newcommand{\edf}{\end{Def}}
\newcommand{\bexm}{\begin{exm}}
\newcommand{\eexm}{\end{exm}}
\newcommand{\blem}{\begin{lem}}
\newcommand{\elem}{\end{lem}}
\newcommand{\bpf}{\begin{proof}}
\newcommand{\epf}{\end{proof}}
\newcommand{\bcor}{\begin{cor}}
\newcommand{\ecor}{\end{cor}}
\newcommand{\ba}{\begin{array}}
\newcommand{\ea}{\end{array}}
\newcommand{\bea}{\begin{eqnarray}}
\newcommand{\eea}{\end{eqnarray}}

\newcommand{\brem}{\begin{remark}}
\newcommand{\erem}{\end{remark}}

\newsavebox{\tablebox}

\begin{document}

\begin{center}
{\large  \bf A linear algebraic proof of the Laplacian spread conjecture}
\footnote {Supported by NSFC (Nos.  12071209, 12231009).}\\

 \vskip 0.8cm
 {\small  Bocheng Li \ \ and\ \ Qingzhong Ji \footnote{Corresponding author.\\ \indent E-mail addresses: borchenli@smail.nju.edu.cn (B. Li),\; qingzhji@nju.edu.cn (Q. Ji)}}\\
{\small School of Mathematics, Nanjing University, Nanjing
210093, P.R.China}

\vskip 3mm
\end{center}

{\bf Abstract:} {\small For a graph $G,$ let $\alpha(G)$ denote its second smallest Laplacian eigenvalue. The Laplacian Spread Conjecture is that $\alpha(G)+\alpha(\overline{G}) \geq 1,$ where $\overline{G}$ is the complement of $G.$ In this article, we provide a new proof of the Laplacian spread conjecture by means of linear algebra, which is more concise. }



\section{\bf Introduction}\label{1}

All concepts used in this paper can be found in \cite{BM} \cite{GS} and in the articles  cited below, unless defined otherwise. Let $G$ be an undirected and unweighted simple graph with vertex set $V(G)=\{1,2,\ldots , n\}$ and edge set $E(G)=\{(i,j)| i,j \in V(G)\}.$ For a vertex $i$, we denote by $d_i$ its degree, i.e., the number of edges incident to $i$. The adjacency matrix of $G$ denoted by $A(G) = [a_{ij}],$ is a square matrix whose entries are indexed by $n \times n$ and $a_{ij} = 1$ if $\{i,j\} \in E(G)$ and $0$ otherwise, where $n$ is the order of $G$, i.e. the number of the vertices. Obviously, $A(G)^{T}=A(G).$ We denote the Laplacian matrix of $G$ by $L = D - A,$ where $A$ is the adjacency matrix of $G$ and $D$ is a diagonal matrix of vertex degrees, i.e. whose $ith$ diagonal entry is $d(i).$

Let $x , y \in \mathbb{R}^n$ are vectors, we denote their inner product by $<x,y> =\sum_{1 \leq i \leq n}x_{i}x_{i}.$ Denote the norm of $x$ by $\parallel x \parallel = \sqrt{<x,x>}.$ We say two vectors $x , y \in R^{n}$ are orthonormal if $\parallel x \parallel = \parallel y \parallel =1$ and $<x,y> =0.$ 

We denote the Laplacian eigenvalues of $G$ by
           $$0= \mu_{n}(G) \leq \mu_{n-1}(G) \leq \cdots \leq \mu_{1}(G).$$
We denote the complement of $G$ by $\overline{G}$. It is well-known that the eigenvalues of $L_{\overline{G}}$ are
      $$0= \mu_{n}(\overline{G}) \leq n- \mu_{1}(G) \leq \cdots \leq n- \mu_{n-1}(G).$$
In \cite{F1}, Fiedler defined the algebraic connectivity of the graph $G$ by $\alpha(G) = \mu_{n-1}(G).$ A related and useful quantity is $\beta(G) = \mu_{1}(G) = n - \alpha(\overline{G}),$ where $\overline{G}$ denotes the complement of the graph $G.$

The Laplacian spread of a graph $G$ is defined as $\beta(G)-\alpha(G).$ Clearly, $\beta(G)-\alpha(G) \leq n.$ In \cite{LY} and \cite{ZSH}, it was conjectured that $n$ can be replaced with $n-1.$ M. Zhai et al.\cite{ZSH} formally proposed the following conjecture, which we refer to as the Laplacian spread conjecture, denoted by LSC. This conjecture has several equivalent formulations, among which the most concise and most frequently studied is the one in terms of algebraic connectivity: \\
\textbf{Laplacian Spread Conjecture (LSC):} For any graph $G$ of order $n \geq 2,$ the following holds:
          $$\beta(G)-\alpha(G) \leq n-1,$$
or equivalently $\alpha(G) + \alpha(\overline{G}) \geq 1,$ with equality if and only if $G$ or $\overline{G}$ is isomorphic to the join of an isolated vertex and a disconnected graph of order $n-1.$

Over a decade of effort, researchers gradually accumulated evidence by verifying the conjecture for various special classes of graphs. The Laplacian spread conjecture has been established for trees \cite{FXWL}, unicyclic graphs \cite{BTF}, bicyclic graphs \cite{LW}, tricyclic graphs \cite{CW}, cacti \cite{L}, quasi-tree graphs \cite{XM}, graphs with diameter not equal to 3 \cite{ZSH}, bipartite graphs \cite{AT}, K3-free graphs \cite{CD} and $t$-quasi-regular graphs of order $n$ when $t \leq \sqrt{n-3+\frac{2}{n}}$ \cite{CD}. In 2019, B.~Afshari and S.~Akbari \cite{AA} gave another characterisation of the Laplacian spread conjecture by establishing an equivalence between graphs and vectors. More precisely, they proved that the following two statements are equivalent:

(\romannumeral1) For every graph \(G\) on \(n\ge 2\) vertices,
\[
\alpha(G)+\alpha(\overline G)\ge 1.
\]

(\romannumeral2) For any two orthogonal unit vectors \(x,y\in\mathbb R^n\) with \(x\perp e,\ y\perp e\),
\[
\|\nabla_x-\nabla_y\|^2 \ge 2,
\]
where \(\nabla_x\in\mathbb R^{\binom n2}\) is the vector of length \(\frac{n(n-1)}2\) whose entry at position \(ij\) (with \(i<j\)) is \(|x_i-x_j|\), and \(e\) is the all-one vector.

Subsequently, many researchers attempted to prove the conjecture by establishing statement (ii) above, but no breakthrough was achieved. Finally, in 2021, M.~Einollahzadeh and M.~M.~Karkhaneei \cite{EK} gave a proof of the Laplacian spread conjecture. Their proof relied on a detailed structural analysis of graphs, introducing the powerful tool of effective resistance together with Fiedler's lemma \cite{F1}. However, their proof is somewhat involved. In this paper, we give a new short proof of the Laplacian spread conjecture using only linear algebraic methods. For more information on the Laplacian spread conjecture, we refer the reader to \cite{AA, AAMM, AGRR, BTF}.

\section{\bf Preliminaries and notations}\label{2}

Let $A=(a_{ij})$ be a real symmetric $0$-$1$ matrix of order $n$ with $a_{ii}=0$. Denote
$$
N_i(A)=\{j: a_{ij}=1\}, \qquad d_i(A)=|N_i(A)|.
$$
Let $e \in \mathbb{R}^n$ be the all-one vector, i.e., $e=(1,1, \cdots, 1)^{T}$, we denote the Laplacian matrix of $A$ by 
$$
L_A=\operatorname{diag}(Ae)-A.
$$ 

For any $z=(z_1,\ldots,z_n)^T\in\mathbb{R}^n$,
$$
z^T L_A z=\sum_{\substack{1\leq i<j\leq n \\ a_{ij}=1}}(z_i-z_j)^2.
$$
For simplicity, we denote $z^T L_A z$ by $E_A(z)$.

Define a $0$-$1$ matrix $A^{[2]}=(a_{ij}^{[2]})$ by
$$
a_{ij}^{[2]}=
\begin{cases}
1, & i\neq j,\ a_{ij}=0,\ \sum_{k=1}^n a_{ik}a_{kj}>0,\\[2mm]
0, & \text{otherwise.}
\end{cases}
$$
In other words, $A^{[2]}$ records precisely those off-diagonal zero entries of $A$ that share a common support after one step of matrix multiplication.

Let \(A\) be an \(n\times n\) symmetric \(0\)-\(1\) matrix, and set \(\bar A=J-I-A\), where \(J\) is the all-one matrix and \(I\) is the identity matrix. The Laplacian eigenvalues of \(A\) are denoted by
$$
0=\mu_n(L_A)\le \mu_{n-1}(L_A)\le \cdots \le \mu_1(L_A),
$$
where \(\mu_i(L_A)\) is the \(i\)-th largest eigenvalue of \(L_A\). In particular, we write
\[
\beta(A)=\mu_1(L_A),\qquad \alpha(A)=\mu_{n-1}(L_A).
\]

Since
$$
L_A+L_{\bar A}=nI-J,
$$
hence
$$
\beta(A)=n-\alpha(\bar A).
$$

We now state the matrix version of the Laplacian spread conjecture:

For any \(n\times n\) symmetric \(0\)-\(1\) matrix \(A\) with zero diagonal (\(n\ge 2\)), and for its Laplacian \(L_A\), we have
\[
\alpha(A)+\alpha(\bar A)\ge 1,
\]
or equivalently,
\[
\beta(A)-\alpha(A)\le n-1.
\]

In this paper, we prove the matrix version of the Laplacian spread conjecture using linear algebraic methods, and consequently establish the graph version as well.

\section{\bf Main results}\label{3}

In this chapter, we present a more concise proof of the Laplacian spread conjecture based on linear algebra.

\begin{thm}\label{thm3.1}
Let \(A=(a_{ij})\) be an \(n\times n\) real symmetric \(0\)-\(1\) matrix with \(a_{ii}=0\) for \(i\in\{1,2,\ldots,n\}\). Then for any \(z\in\mathbb R^n\), we have
$$
(n-1)E_A(z)\geq E_{A^{[2]}}(z).
$$
\end{thm}

\begin{proof} Consider \(A\) as the adjacency matrix of a simple graph \(G\) on the vertex set
\[
V=\{1,\ldots,n\}.
\]
For \(i\ne j\), denote
\[
c_{ij}=|N_i\cap N_j|.
\]
Then
\[
a^{[2]}_{ij}=1
\quad\Longleftrightarrow\quad
a_{ij}=0\ \text{and}\ c_{ij}>0.
\]

First compute \(L_A^2\). Since
\[
L_A=\operatorname{diag}(d_1,\ldots,d_n)-A,
\]
we have
\[
(L_A^2)_{ij}
=
\begin{cases}
d_i^2+d_i,& i=j,\\
-a_{ij}(d_i+d_j)+c_{ij},& i\ne j.
\end{cases}
\]
Also, because \(L_A e=0\), we get
\[
L_A^2e=0.
\]
Hence \(L_A^2\) is a real symmetric matrix with zero row sums. For any real symmetric matrix
\(M=(m_{ij})\) with zero row sums, the identity holds:
\[
z^TMz=\sum_{1\le i<j\le n}(-m_{ij})(z_i-z_j)^2.
\]
Substituting \(M=L_A^2\), we obtain
\[
\begin{aligned}
z^TL_A^2z
={}&
\sum_{\substack{1\le i<j\le n\\ a_{ij}=1}}
(d_i+d_j-c_{ij})(z_i-z_j)^2 \\
&-
\sum_{\substack{1\le i<j\le n\\ a_{ij}=0}}
c_{ij}(z_i-z_j)^2.
\end{aligned}
\]
On the other hand,
\[
z^TL_A^2z=\|L_Az\|^2\ge 0.
\]
Thus
\begin{equation}\label{eq:A2-basic}
\sum_{\substack{1\le i<j\le n\\ a_{ij}=0}}
c_{ij}(z_i-z_j)^2
\le
\sum_{\substack{1\le i<j\le n\\ a_{ij}=1}}
(d_i+d_j-c_{ij})(z_i-z_j)^2.
\end{equation}
When \(a_{ij}=1\),
\[
d_i+d_j-c_{ij}
=
|N_i\cup N_j|.
\]
Therefore, from \eqref{eq:A2-basic} we get
\begin{equation}\label{eq:A2-basic-union}
\sum_{\substack{1\le i<j\le n\\ a_{ij}=0}}
c_{ij}(z_i-z_j)^2
\le
\sum_{\substack{1\le i<j\le n\\ a_{ij}=1}}
|N_i\cup N_j|(z_i-z_j)^2.
\end{equation}
Also, from \(a^{[2]}_{ij}=1\) we have \(a_{ij}=0\) and \(c_{ij}\ge 1\), so
\begin{equation}\label{eq:A2-left}
E_{A^{[2]}}(z)
\le
\sum_{\substack{1\le i<j\le n\\ a_{ij}=0}}
c_{ij}(z_i-z_j)^2.
\end{equation}

If for every edge \(ij\) of \(G\),
\[
N_i\cup N_j\ne V,
\]
then
\[
|N_i\cup N_j|\le n-1
\]
for each edge \(ij\). From \eqref{eq:A2-basic-union} and
\eqref{eq:A2-left} we immediately obtain
\[
E_{A^{[2]}}(z)\le (n-1)E_A(z).
\]
Therefore, it suffices to handle the case where there exists an edge \(ij\) satisfying
\[
N_i\cup N_j=V.
\]
Call such an edge a dominating edge.

We now proceed by induction on \(n\). The case \(n\le 2\) is trivial. Assume the statement holds for all matrices of order less than \(n\). Suppose \(A\) is a counterexample of order \(n\); then there exists \(z\in\mathbb R^n\) such that
\begin{equation}\label{eq:A2-counter}
(n-1)E_A(z)<E_{A^{[2]}}(z).
\end{equation}
From the previous discussion, \(G\) must contain a dominating edge.

If all dominating edges \(ij\) satisfy \(z_i=z_j\), then on the right-hand side of
\eqref{eq:A2-basic-union}, all terms corresponding to edges with
\[
|N_i\cup N_j|=n
\]
are zero, and for the remaining edges we have
\[
|N_i\cup N_j|\le n-1.
\]
Thus we still have
\[
E_{A^{[2]}}(z)\le (n-1)E_A(z),
\]
contradicting \eqref{eq:A2-counter}. Hence we can choose a dominating edge \(ij\) such that
\begin{equation}\label{eq:dom-edge}
N_i\cup N_j=V, \ \ z_i\ne z_j.
\end{equation}

Define
\[
X=N_i\setminus (N_j\cup\{j\}), \ \ 
Y=N_j\setminus (N_i\cup\{i\}), \ \ 
Z=N_i\cap N_j.
\]
Then
\[
V=\{i,j\}\sqcup X\sqcup Y\sqcup Z.
\]

If \(X=\varnothing\), then from \(N_i\cup N_j=V\) we get
\[
N_j=V\setminus\{j\}.
\]
Indeed, if \(v\ne j\), then when \(v=i\) clearly \(v\in N_j\); when
\(v\ne i,j\), from \(N_i\cup N_j=V\) we have \(v\in N_i\) or
\(v\in N_j\). If \(v\in N_i\), since \(X=\varnothing\), we must have
\(v\in N_j\). Hence every \(v\ne j\) belongs to \(N_j\), i.e.
\[
N_j=V\setminus\{j\}.
\]
Thus \(j\) is a universal vertex. In this case, \(A^{[2]}\) has no edges incident to \(j\), so
\[
E_{A^{[2]}}(z)
\le
\sum_{\substack{r<s\\ r,s\ne j}}(z_r-z_s)^2.
\]
For the \(N=n-1\) numbers
\[
u_r=z_r-z_j \ \ (r\ne j),
\]
using the identity
\[
\sum_{r<s}(u_r-u_s)^2
=
N\sum_r u_r^2-\left(\sum_r u_r\right)^2
\le
N\sum_r u_r^2,
\]
we get
\[
\sum_{\substack{r<s\\ r,s\ne j}}(z_r-z_s)^2
\le
(n-1)\sum_{r\ne j}(z_r-z_j)^2.
\]
Since \(j\) is universal,
\[
\sum_{r\ne j}(z_r-z_j)^2\le E_A(z).
\]
Hence
\[
E_{A^{[2]}}(z)\le (n-1)E_A(z),
\]
contradicting \eqref{eq:A2-counter}. Therefore
\[
X\ne\varnothing.
\]
Similarly,
\[
Y\ne\varnothing.
\]

Let \(A_1\) be the adjacency matrix of the subgraph of \(G\) obtained by keeping only the edges incident to \(i\) or \(j\). That is, the edges of \(A_1\) are exactly
\[
ij,\qquad ir\ (r\in X\cup Z),\qquad js\ (s\in Y\cup Z).
\]
Let \(B\) be the principal submatrix of \(A\) obtained by deleting the \(i\)-th and \(j\)-th rows and columns.
If \(z_1\) denotes the restriction of \(z\) to the vertex set \(V\setminus\{i,j\}\), then
\[
E_A(z)=E_{A_1}(z)+E_B(z_1).
\]

We first prove the inequality
\begin{equation}\label{eq:A1-key}
(n-1)E_{A_1}(z)\ge E_{A_1^{[2]}}(z).
\end{equation}

Set
\[
m=|X|,\qquad \ell=|Y|,\qquad h=|Z|.
\]
From the above,
\[
m,\ell\ge 1,\qquad n=m+\ell+h+2.
\]
Let
\[
L=z_i-z_j.
\]
For \(r\in X\), write
\[
\alpha_r=z_r-z_i,
\]
for \(s\in Y\), write
\[
\beta_s=z_s-z_j,
\]
and for \(t\in Z\), write
\[
\gamma_t=z_t-z_i.
\]
Thus
\[
z_t-z_j=\gamma_t+L.
\]
Also define
\[
U=\sum_{r\in X}\alpha_r,\qquad
T=\sum_{s\in Y}\beta_s,\qquad
S=\sum_{t\in Z}\gamma_t,\qquad
Q=\sum_{t\in Z}\gamma_t^2.
\]
If \(Z=\varnothing\), we set
\[
S=Q=0.
\]

From the definition, we have
\[
\begin{aligned}
E_{A_1}(z)
={}&
L^2+\sum_{r\in X}\alpha_r^2
+\sum_{s\in Y}\beta_s^2
+\sum_{t\in Z}\gamma_t^2
+\sum_{t\in Z}(\gamma_t+L)^2.
\end{aligned}
\]

Next, we estimate \(E_{A_1^{[2]}}(z)\). In the graph \(A_1\),
\[
N_i(A_1)=X\cup Z\cup\{j\},
\qquad
N_j(A_1)=Y\cup Z\cup\{i\}.
\]
Thus the edges of \(A_1^{[2]}\) can only arise from non-edges inside these two neighbourhoods. If two vertices both lie in
\(N_i(A_1)\) and are nonadjacent in \(A_1\), they have a common neighbour
\(i\); similarly, if two vertices both lie in \(N_j(A_1)\) and are nonadjacent in \(A_1\), they have a common neighbour \(j\).

Notice that the two large sums below may include pairs that do not belong to \(A_1^{[2]}\).
For example, pairs that are already edges of \(A_1\) cannot belong to \(A_1^{[2]}\); meanwhile,
pairs inside \(Z\) belong to both \(N_i(A_1)\) and \(N_j(A_1)\), hence would be counted twice;
also, for each \(t\in Z\), the pairs \(\{i,t\}\) and \(\{j,t\}\) are already edges of
\(A_1\), so they cannot be counted in \(A_1^{[2]}\). Therefore we have the upper bound
\begin{equation}\label{eq:A1-upper}
\begin{aligned}
E_{A_1^{[2]}}(z)
\le{}&
\sum_{\{u,v\}\subseteq N_i(A_1)}(z_u-z_v)^2
+
\sum_{\{u,v\}\subseteq N_j(A_1)}(z_u-z_v)^2  \\
&-
\sum_{\{u,v\}\subseteq Z}(z_u-z_v)^2
-
\sum_{t\in Z}(z_j-z_t)^2
-
\sum_{t\in Z}(z_i-z_t)^2.
\end{aligned}
\end{equation}
This is an upper bound rather than equality because the first two large sums may contain other pairs already adjacent in \(A_1\); but the listed duplicate terms and edge terms must be subtracted, so
\eqref{eq:A1-upper} remains valid.

For any finite set \(W\) and any real number \(a\), we have the identity
\[
\sum_{\{u,v\}\subseteq W}(z_u-z_v)^2
=
|W|\sum_{u\in W}(z_u-a)^2
-
\left(\sum_{u\in W}(z_u-a)\right)^2.
\]
Taking \(W=N_i(A_1),a=z_i\) and \(W=N_j(A_1),a=z_j\), respectively, yields
\[
\begin{aligned}
\sum_{\{u,v\}\subseteq N_i(A_1)}(z_u-z_v)^2
={}&
(m+h+1)
\left(
\sum_{r\in X}\alpha_r^2+Q+L^2
\right)
-\bigl(U+S-L\bigr)^2,
\\
\sum_{\{u,v\}\subseteq N_j(A_1)}(z_u-z_v)^2
={}&
(\ell+h+1)
\left(
\sum_{s\in Y}\beta_s^2
+\sum_{t\in Z}(\gamma_t+L)^2
+L^2
\right)  \\
&-
\bigl(T+S+(h+1)L\bigr)^2.
\end{aligned}
\]
Moreover,
\[
\sum_{\{u,v\}\subseteq Z}(z_u-z_v)^2=hQ-S^2,
\]
and
\[
\sum_{t\in Z}(z_i-z_t)^2=Q,\qquad
\sum_{t\in Z}(z_j-z_t)^2
=
\sum_{t\in Z}(\gamma_t+L)^2.
\]

Substituting these expressions into \eqref{eq:A1-upper} and subtracting from
\((n-1)E_{A_1}(z)\), after simplification we obtain
\begin{equation}\label{eq:A1-difference}
\begin{aligned}
&(n-1)E_{A_1}(z)-E_{A_1^{[2]}}(z)  \\
\ge{}&
\ell\sum_{r\in X}\alpha_r^2
+m\sum_{s\in Y}\beta_s^2
+\bigl(U+S-L\bigr)^2
+\bigl(T+S+(h+1)L\bigr)^2  \\
&+
(h+m+\ell+2)Q
-S^2
+
(hm-1)L^2
+
2(m+1)LS.
\end{aligned}
\end{equation}

We now prove that the right-hand side of \eqref{eq:A1-difference} is nonnegative. By Cauchy's inequality,
\[
\sum_{r\in X}\alpha_r^2\ge \frac{U^2}{m},
\qquad
\sum_{s\in Y}\beta_s^2\ge \frac{T^2}{\ell}.
\]
Thus, for any real number \(C\),
\[
\ell\sum_{r\in X}\alpha_r^2+(U+C)^2
\ge
\frac{\ell}{m}U^2+(U+C)^2
\ge
\frac{\ell}{m+\ell}C^2.
\]
Taking \(C=S-L\), we get
\[
\ell\sum_{r\in X}\alpha_r^2+\bigl(U+S-L\bigr)^2
\ge
\frac{\ell}{m+\ell}(S-L)^2.
\]
Similarly, for any real number \(D\),
\[
m\sum_{s\in Y}\beta_s^2+(T+D)^2
\ge
\frac{m}{\ell}T^2+(T+D)^2
\ge
\frac{m}{m+\ell}D^2.
\]
Taking \(D=S+(h+1)L\), we get
\[
m\sum_{s\in Y}\beta_s^2+\bigl(T+S+(h+1)L\bigr)^2
\ge
\frac{m}{m+\ell}\bigl(S+(h+1)L\bigr)^2.
\]
Substituting into \eqref{eq:A1-difference} yields
\begin{equation}\label{eq:A1-lower}
\begin{aligned}
&(n-1)E_{A_1}(z)-E_{A_1^{[2]}}(z)  \\
\ge{}&
\frac{\ell}{m+\ell}(S-L)^2
+
\frac{m}{m+\ell}\bigl(S+(h+1)L\bigr)^2 \\
&+
(h+m+\ell+2)Q
-S^2
+
(hm-1)L^2
+
2(m+1)LS.
\end{aligned}
\end{equation}

If \(h=0\), then \(S=Q=0\). Note that
\[
(hm-1)L^2=-L^2.
\]
From \eqref{eq:A1-lower} we get
\[
\begin{aligned}
(n-1)E_{A_1}(z)-E_{A_1^{[2]}}(z)
&\ge
\frac{\ell}{m+\ell}L^2
+
\frac{m}{m+\ell}L^2
-L^2 \\
&=0.
\end{aligned}
\]
Thus \eqref{eq:A1-key} holds.

If \(h\ge 1\), then
\[
Q\ge \frac{S^2}{h}.
\]
Substituting this into \eqref{eq:A1-lower} and simplifying yields
\[
\begin{aligned}
&(n-1)E_{A_1}(z)-E_{A_1^{[2]}}(z)\\
\ge{}&
\frac{h+m+\ell+2}{h(m+\ell)}
\left(
h^2mL^2+2hmLS+(m+\ell)S^2
\right).
\end{aligned}
\]
The quadratic form in \(L,S\) inside the parentheses is
\[
h^2mL^2+2hmLS+(m+\ell)S^2,
\]
whose associated matrix is
\[
\begin{pmatrix}
h^2m & hm\\
hm & m+\ell
\end{pmatrix}.
\]
Its leading principal minor is \(h^2m>0\), and its determinant is
\[
h^2m(m+\ell)-h^2m^2=h^2m\ell>0.
\]
Since \(m,\ell,h\ge 1\), this quadratic form is positive definite. Therefore the right-hand side is nonnegative, and hence
\[
(n-1)E_{A_1}(z)\ge E_{A_1^{[2]}}(z).
\]
This proves \eqref{eq:A1-key}.

Now return to the induction. Since \(B\) has order \(n-2\), the induction hypothesis gives
\[
(n-3)E_B(z_{1})\ge E_{B^{[2]}}(z_{1}),
\]
where \(z_{1}\) is the vector obtained from \(z\) by deleting the entries corresponding to \(\{i,j\}\).
Consequently,
\begin{equation}\label{eq:B-bound}
(n-1)E_B(z_{1})\ge E_{B^{[2]}}(z_{1}),
\end{equation}
because \(E_B(z_{1})\ge 0\).

Finally, compare \(A^{[2]}\) with \(A_1^{[2]}\) and \(B^{[2]}\). Let \(uv\) be an edge of
\(A^{[2]}\). Then \(uv\) is not an edge of \(A\), but \(u,v\)
have a common neighbour in \(A\).

If \(u\in X,\ v\in Y\), then
\[
u\sim i,\quad u\not\sim j,\qquad
v\sim j,\quad v\not\sim i.
\]
Thus neither \(i\) nor \(j\) is a common neighbour of \(u,v\). Hence any common neighbour of \(u,v\) in \(A\) must lie in
\[
V\setminus\{i,j\}.
\]
Since \(uv\) is not an edge of \(A\), and \(B\) is the induced subgraph of \(A\) on
\(V\setminus\{i,j\}\), it follows that \(uv\) is an edge of \(B^{[2]}\).

Now consider the case where \(uv\) is not a pair with one endpoint in \(X\) and the other in \(Y\). Since
\(uv\) is not an edge of \(A\), from the decomposition
\[
V=\{i,j\}\sqcup X\sqcup Y\sqcup Z
\]
and the definition of \(A_1\), the remaining possibilities force
\[
\{u,v\}\subseteq X\cup Z\cup\{j\},
\]
or
\[
\{u,v\}\subseteq Y\cup Z\cup\{i\}.
\]
In the first case,
\[
u,v\in N_i(A_1);
\]
in the second case,
\[
u,v\in N_j(A_1).
\]
Since \(A_1\) is a subgraph of \(A\) and \(uv\) is not an edge of \(A\), it is certainly not an edge of \(A_1\). Hence \(uv\) is an edge of \(A_1^{[2]}\).

In particular, a pair of the form \(jx\) with \(x\in X\), if it belongs to \(A^{[2]}\), also belongs to
\(A_1^{[2]}\), because
\[
j,x\in N_i(A_1).
\]
Similarly, a pair of the form \(iy\) with \(y\in Y\), if it belongs to \(A^{[2]}\), also belongs to
\(A_1^{[2]}\), because
\[
i,y\in N_j(A_1).
\]

Therefore every edge of \(A^{[2]}\) is covered by either \(A_1^{[2]}\) or \(B^{[2]}\).
Overlaps are allowed, since we only need an upper bound on the energy. Hence
\begin{equation}\label{eq:A2-cover}
E_{A^{[2]}}(z)
\le
E_{A_1^{[2]}}(z)+E_{B^{[2]}}(z_{1}).
\end{equation}

Combining
\[
E_A(z)=E_{A_1}(z)+E_B(z_{1}),
\]
with \eqref{eq:A1-key}, \eqref{eq:B-bound}, and \eqref{eq:A2-cover}, we get
\[
\begin{aligned}
E_{A^{[2]}}(z)
&\le
E_{A_1^{[2]}}(z)+E_{B^{[2]}}(z_{1})\\
&\le
(n-1)E_{A_1}(z)+(n-1)E_B(z_{1})\\
&=
(n-1)E_A(z).
\end{aligned}
\]
This contradicts the counterexample assumption \eqref{eq:A2-counter}. Therefore the case of order \(n\) holds.
By induction, the theorem is proved.
\end{proof}\qed

\begin{cor}\label{cor3.2}
Let $A$ be an $n\times n$ symmetric $0$-$1$ matrix with \(a_{ii}=0\) for \(i\in\{1,\dots,n\}\), and let \(z\in\mathbb R^n\). Define
$$
R(A)=\left\{i: \exists\, j\neq i,\ a_{ij}=0,\ \sum_{k=1}^n a_{ik}a_{kj}=0\right\}.
$$
If \(z\perp \mathbf{1}\) and \(\|z\|=1\), then
$$
E_A(z) \ge 1-\frac{|R(A)|}{n}.
$$
\end{cor}

\begin{proof} Let
$$
C=A+A^{[2]}.
$$
By Theorem \ref{thm3.1},
$$
nE_A(z)\ge E_C(z).
$$
When \(i\notin R(A)\), all off-diagonal entries in the \(i\)-th row of \(C\) are equal to \(1\). Let
$$
t=n-|R(A)|.
$$
Then \(C\) has at least \(t\) rows with the above property. Let \(W\) be the set of remaining indices, so \(|W|=n-t\). Since
$$
\sum_{1\le r<s\le n}(z_r-z_s)^2
= n\|z\|^2 - \langle z,\mathbf{1}\rangle^2 = n,
$$
and all possible missing off-diagonal positions can only lie in \(W\times W\), we have
$$
\begin{aligned}
E_C(z)
&\ge n - \sum_{\substack{r<s\\ r,s\in W}}(z_r-z_s)^2 \\
&= n - \left(|W|\sum_{r\in W}z_r^2 - \left(\sum_{r\in W}z_r\right)^2\right) \\
&\ge n - |W| = t.
\end{aligned}
$$
Therefore,
$$
E_A(z)\ge \frac{t}{n} = 1-\frac{|R(A)|}{n}.
$$
\end{proof}\qed

\medskip

We now prove the matrix version of the Laplacian spread conjecture.

\begin{thm}[Laplacian spread]\label{thm3.3}
Let \(A\) be an \(n\times n\) symmetric \(0\)-\(1\) matrix with zero diagonal, \(n\ge2\), and let \(L_A\) be its Laplacian. Then
$$
\alpha(A)+\alpha(\bar A)\ge 1,
$$
or equivalently,
$$
\beta(A)-\alpha(A)\le n-1.
$$
\end{thm}

\begin{proof} By Corollary \ref{cor3.2},
$$
\alpha(A)\ge 1-\frac{|R(A)|}{n},\qquad
\alpha(\bar A)\ge 1-\frac{|R(\bar A)|}{n}.
$$
We next show that
$$
R(A)\cap R(\bar A)=\varnothing.
$$
If \(i\in R(A)\), then there exists \(j\ne i\) such that
$$
a_{ij}=0,\qquad \sum_{k=1}^n a_{ik}a_{kj}=0.
$$
Take any \(r\ne i\). If \(a_{ir}=0\), then \(\bar a_{ir}=1\); if \(a_{ir}=1\), then from \(\sum_k a_{ik}a_{kj}=0\) we get \(a_{rj}=0\), hence
$$
\bar a_{ij}\,\bar a_{jr}=1.
$$
Thus for every \(r\ne i\), either \(\bar a_{ir}=1\) or \(\sum_k \bar a_{ik}\bar a_{kr}>0\). This precisely means that \(i\notin R(\bar A)\). Hence \(R(A)\) and \(R(\bar A)\) are disjoint, so
$$
|R(A)|+|R(\bar A)|\le n.
$$
Consequently,
$$
\alpha(A)+\alpha(\bar A)
\ge 2-\frac{|R(A)|+|R(\bar A)|}{n}
\ge 1.
$$
This proves the theorem.
\end{proof}\qed

\medskip

If we take \(G\) to be the graph with adjacency matrix \(A\), then Theorem \ref{thm3.3} directly implies the Laplacian spread conjecture for graphs.

\begin{cor}\label{cor3.4}
For every graph \(G\) on \(n\ge2\) vertices, the following inequality holds:
$$
\alpha(G)+\alpha(\overline G)\ge 1,
$$
or equivalently,
$$
\beta(G)-\alpha(G)\le n-1.
$$
\end{cor}

\vspace{5pt}
\textbf{Remark:} Our proof of the Laplacian spread conjecture uses only the quadratic form of real symmetric matrices, Cauchy's inequality, complementary matrices, and the Rayleigh quotient characterisation of eigenvalues, without any complicated structural analysis of graphs. This makes the proof significantly more concise than that of M.~Einollahzadeh and M.~M.~Karkhaneei (2021).

\


\section{\bf Declaration of competing interest}\label{4}

There is no competing interest.


\begin{thebibliography}{99}

\bibitem{AA}B. Afshari, S. Akbari, Some results on the Laplacian spread conjecture, Linear Algebra Appl. 574 (2019) 22-29.

\bibitem{AAMM}B. Afshari, S. Akbari, M.J. Moghaddamzadeh, B. Mohar, The algebraic connectivity of a graph and its complement, Linear Algebra Appl. 555 (2018) 157-C162.

\bibitem{AT}F. Ashraf, B. Tayfeh-Rezaie, Nordhaus–Gaddum type inequalities for Laplacian and signless Laplacian eigenvalues, Electron. J. Combin. 21 (3) (2014), Paper 3.6, 13 pp.

\bibitem{AGRR}E. Andrade, H. Gomes, M. Robbiano, J. Rodriguez, Upper bounds on the Laplacian spread of graphs, Linear Algebra Appl. 492 (2016) 26-C37.

\bibitem{BTF}Y.H. Bao, Y.Y. Tan, Y.Z. Fan, The Laplacian spread of unicyclic graphs, Appl. Math. Lett. 22 (2009) 1011-C1015.

 \bibitem{BCEHK} W. Barrett, T. R. Cameron, E. Evans, H. T. Hall, M. Kempton, On the Laplacian spread of digraphs, Linear Algebra and its Applications 664 (2023) 126-C146

\bibitem{BM}J.A. Bondy, U.S.R. Murty, Graph Theory, Springer, 2008.



\bibitem{F1}Miroslav Fiedler, A property of eigenvectors of nonnegative symmetric matrices and its application to graph theory, Czechoslov, Math. J. 25 (1975) 619–633.

\bibitem{EK}M. Einollahzadeh, M.M. Karkhaneei, On the lower bound of the sum of the algebraic connectivity of a graph and its complement, J. Comb. Theory, Ser. B 151 (2021) 235-C249.


\bibitem{GS}G. Strang, Introduction to Linear Algebra, Wellesley-Cambridge Press, 2023.

\bibitem{LY}Z. You, B. Liu, The Laplacian spread of graphs, Czechoslovak Math. J. 62 (137) (2012) 155-C168.

\bibitem{ZSH}M. Zhai, J. Shu, Y. Hong, On the Laplacian spread of graphs, Appl. Math. Lett. 24 (2011) 2097–2101.



\bibitem{FXWL}Z.Y. Fan, J. Xu, Y. Wang, D. Liang, The Laplacian spread of a tree, Discrete Math. Theor. Comput. Sci. 10 (1) (2008) 79–86.


\bibitem{LW}Y. Liu, L. Wang, The Laplacian spread of bicyclic graphs, Adv. Math. (China) 40 (2011) 759–764.

\bibitem{CW}Y. Chen, L. Wang, The Laplacian spread of tricyclic graphs, Electron. J. Combin. 16 (2009), Research Paper 80, 18 pp.

\bibitem{L}Y. Liu, The Laplacian spread of cactuses, Discrete Math. Theor. Comput. Sci. 12 (2010) 35–40.

\bibitem{XM}Y. Xu, J. Meng, The Laplacian spread of quasi-tree graphs, Linear Algebra Appl. 435 (2011) 60–66.




\bibitem{CD}X. Chen, K.C. Das, Some results on the Laplacian spread of a graph, Linear Algebra Appl. 505
(2016) 245–260.


\end{thebibliography}
\end{document}